\documentclass[a4paper,12pt]{article}

\usepackage[left=2cm,right=2cm, top=2cm,bottom=3cm,bindingoffset=0cm]{geometry}

\usepackage{verbatim}
\usepackage{amsmath}
\usepackage{amsthm}
\usepackage{amssymb}
\usepackage{delarray}
\usepackage{cite}
\usepackage{hyperref}
\usepackage{mathrsfs}
\usepackage{tikz}
\usetikzlibrary{patterns}
\usepackage{caption}
\DeclareCaptionLabelSeparator{dot}{. }
\newcommand{\al}{\alpha}
\newcommand{\be}{\beta}
\newcommand{\ga}{\gamma}
\newcommand{\de}{\delta}
\newcommand{\la}{\lambda}
\newcommand{\om}{\omega}

\newcommand{\eps}{\varepsilon}
\newcommand{\vv}{\varphi}

\theoremstyle{plain}

\numberwithin{equation}{section}

\newtheorem{thm}{Theorem}[section]
\newtheorem{lem}[thm]{Lemma}
\newtheorem{prop}[thm]{Proposition}
\newtheorem{cor}[thm]{Corollary}

\theoremstyle{definition}

\newtheorem*{ip}{Inverse Problem}
\newtheorem*{dirp}{Direct Problem}

\theoremstyle{remark}

\newtheorem{remark}[thm]{Remark}

\DeclareMathOperator*{\Res}{Res}
\DeclareMathOperator{\dom}{dom}
\DeclareMathOperator{\ind}{ind}

 \allowdisplaybreaks

\begin{document}

\begin{center}
{\Large\bf Uniform stability for the inverse Sturm-Liouville\\[0.2cm] problem with eigenparameter-dependent\\[0.2cm] boundary conditions}
\\[0.5cm]
{\bf Natalia P. Bondarenko}
\end{center}

\vspace{0.5cm}

{\bf Abstract.} We consider a class of self-adjoint Sturm-Liouville problems with rational functions of the spectral parameter in the boundary conditions. The uniform stability for direct and inverse spectral problems is proved for the first time for Sturm-Liouville operator pencils with boundary conditions depending on the eigenparameter. Furthermore, we obtain stability estimates for finite data approximations, which are important from the practical viewpoint. Our method is based on Darboux-type transforms and proving of their Lipschitz continuity.

\medskip


\medskip


\section{Introduction} \label{sec:intr}

Consider the eigenvalue problem $\mathcal L(\sigma, f, F)$:
\begin{gather} \label{eqv}
-(y^{[1]}_{\sigma})' - \sigma(x) y^{[1]}_{\sigma} - \sigma^2(x) y = \lambda y, \quad x \in (0,\pi), \\ \label{bc}
\dfrac{y^{[1]}_{\sigma}(0)}{y(0)} = -f(\la), \quad \dfrac{y^{[1]}_{\sigma}(\pi)}{y(\pi)} = F(\la),
\end{gather}
where $\sigma$ is a real-valued function of $L_2[0,\pi]$, $y^{[1]}_{\sigma} := y' - \sigma y$ is the so-called quasi-derivative, $\la$ is the spectral parameter, $f(\la)$ and $F(\la)$ are rational Herglotz-Nevanlinna functions (see Section~\ref{sec:prelim} for details). Solutions of equation \eqref{eqv} are considered in the domain
\begin{equation} \label{defDsi}
\mathcal D_{\sigma} := \{ y \in W_1^1[0,\pi] \colon y^{[1]}_{\sigma} \in W_1^1[0,\pi] \}.
\end{equation}

Equation \eqref{eqv} is equivalent to the Sturm-Liouville equation 
\begin{equation} \label{eqvStL}
-y'' + q(x) y = \la y, \quad x \in (0, \pi),
\end{equation}
with the potential $q(x) := \sigma'(x)$ of the distribution space $W_2^{-1}[0,\pi]$.

This paper is concerned with the theory of inverse spectral problems, which consist in the reconstruction of differential operators from their spectral data. 
Such kind of problems arise in various physical and engineering applications, when one needs to find unknown medium properties from some measurable data or to design a device with desired characteristics. The most complete results in inverse spectral theory have been obtained for the Sturm-Liouville operators with constant coefficients in the boundary conditions (see the monographs \cite{Lev84, FY01, Mar11, Krav20} and references therein). Eigenvalue problems with the spectral parameter in boundary conditions naturally arise in acoustics \cite{KW77}, quantum mechanics \cite{Gran22}, fluid dynamics \cite{KG15}, and other physical applications \cite{Ful80}.
Inverse spectral theory for self-adjoint Sturm-Liouville problems with rational Herglotz-Nevanlinna functions in the boundary conditions was studied in \cite{BBW02, BBW04, Gul19, Gul20-ann, Gul20-ams, Gul23, YW18}. In particular, Guliyev \cite{Gul19} has obtained the spectral data characterization for the problem \eqref{eqv}--\eqref{bc} with $\sigma \in L_2[0,\pi]$. Furthermore, Freiling and Yurko \cite{FY10, FY12} proposed a method for solving non-self-adjoint inverse Sturm-Liouville problems with arbitrary polynomials of the spectral parameter in the boundary conditions. In recent years, the approach of \cite{FY10} was developed for the case of distribution potential $q \in W_2^{-1}[0,\pi]$ in \cite{CB23, CB24, CB25-jiip}. Some other issues of inverse spectral theory for differential and integro-differential operators with polynomials in the boundary conditions were considered in \cite{Wang12, Yang13} and \cite{Bond19}, respectively.

In recent years, a significant progress has been achieved in investigation of uniform stability for inverse spectral problems. In \cite{SS10, SS13}, Savchuk and Shkalikov have proved the unconditional uniform stability for inverse Sturm-Liouville problems with potentials in the Sobolev spaces $W_2^{\theta}$, $\theta \ge -1$. The case $\theta = 0$ was studied by Hryniv \cite{Hryn11}, who also obtained the results of this kind for the Dirac system \cite{Hryn11-dir}. Furthermore, the uniform stability for inverse problems has been proved for several classes of non-local operators, in particular, for integro-differential operators \cite{But21}, functional-differential operators with constant delay \cite{BD22} and with frozen argument \cite{Kuz23}. The uniform stability of the non-self-adjoint Sturm-Liouville problem was investigated in \cite{Bond24}. However, to the best of the author's knowledge, there are no uniform stability results for the Sturm-Liouville inverse problems with eigenparameter-dependent boundary conditions. Local stability of such problems was studied in \cite{CB24, CB25-jiip, CB25-jde}, but therein only small perturbations of spectral data were considered and the constants in the obtained stability estimates depend on a fixed potential. The investigation of uniform stability requires a different approach. One has to find large sets of spectral data or of problem parameters, on which stability estimates are uniform.

In this paper, we prove the uniform stability of direct and inverse spectral problems for the Sturm-Liouville problem~\eqref{eqv}--\eqref{bc} with Herglotz-Nevanlinna functions in the boundary conditions and with $\sigma$ in the scale of Sobolev spaces $W_2^{\al}[0,\pi]$. As spectral data, we use the eigenvalues $\{ \la_n \}_{n \ge 1}$ and the norming constants $\{ \ga_n \}_{n \ge 1}$ according to the problem statement in \cite{Gul19}.
The most challenging issue in studying the uniform stability is to describe such sets, on which stability estimates hold uniformly. We establish the correspondence between the sets $\mathcal P_{Q,\de}$ of the problem parameters $(\sigma, f, F)$ and the sets $\mathcal B_{R,\eps}$ of the spectral data (see Section~\ref{sec:main} for details) and obtain the uniform stability on these sets. Analogously to the previous works by Savchuk and Shkalikov \cite{SS10, SS13} and by Hryniv \cite{Hryn11}, the set $\mathcal B_{R,\eps}$ of spectral data is bounded by the upper constraint $R$, as well as by the lower constraint $\eps > 0$ on the ``gap'' between neighboring eigenvalues $\la_n$ and $\la_{n+1}$ and on the norming constant $\ga_n$. A feature of our study is that the set $\mathcal P_{Q,\de}$ of problem parameters is similarly bounded not only by an upper constraint $Q$ but also by a lower constraint $\de > 0$ related to the Herglotz-Nevanlinna functions $f(\la)$ and $F(\la)$. For the Dirichlet or the Robin boundary conditions, this lower constraint is unnecessary, but it is crucial for forming a closed set of rational functions. Our proof method is based on the Darboux-type transforms that were constructed by Guliyev \cite{Gul19}. Those transforms allow us, by using finite changes of spectral data, reduce the boundary value problem \eqref{eqv}--\eqref{bc} to the similar problem with the Dirichlet boundary conditions. We prove the Lipschitz continuity of Guliyev's transforms and then transfer the results by Savchuk and Shkalikov \cite{SS10} to the problem with eigenparameter-dependent boundary conditions.

Applying our main results, we obtain the uniform stability estimates for the solution of the inverse problem by the finite spectral data $\{ \la_n, \ga_n \}_{n = 1}^m$. This kind of estimates is important from the practical viewpoint, since in applications only a finite amount of data is usually available. Therefore, a number of research works are focused on stability of finite data approximations for the Sturm-Liouville operators on a finite interval \cite{Ryab73, MW05, SS14, Sav16, GMXA23}, on the half-line \cite[Section~5]{Mar11}, and on the line \cite{Akt87, Hit00}. Furthermore, this aspect caused interest for the transmission inverse eigenvalue problem \cite{XY20, XGY22} and for inverse resonance problems \cite{MSW09, Bled12}. In this paper, we consider two cases for the problem \eqref{eqv}--\eqref{bc}, when finite spectral data are known (i) precisely, (ii) with error at most $\eps$. For the both cases, estimates for finite data approximations are readily deduced from our main theorems on the uniform stability.

The paper is organized as follows. Section~\ref{sec:prelim} contains necessary preliminaries about the Sobolev spaces $W_2^{\al}$ and rational Herglotz-Nevanlinna functions. In Section~\ref{sec:main}, we introduce the spectral data and present the main results. In Section~\ref{sec:aux}, auxiliary lemmas about the Lipschitz continuity are proved for some characteristics of the problem $\mathcal L(\sigma, f, F)$. In Section~\ref{sec:trans}, we consider the Darboux-type transforms from \cite{Gul19} and study their Lipschitz continuity. Section~\ref{sec:proofs} contains the proofs of our main theorems.

\section{Preliminaries} \label{sec:prelim}

\subsection{Spaces $W_2^{\alpha}$ and $l_2^{\alpha}$}

Throughout this paper, we denote by $L_p[0,\pi]$ and $W_p^k[0,\pi]$ for $p \ge 1$ and integer $k \ge 0$ the real Lebesgue and Sobolev spaces, respectively, with the corresponding norms:
$$
\| u \|_{L_p[0,\pi]} = \left(\int_0^{\pi} |u(x)|^p \, dx\right)^{1/p}, \quad \| u \|_{W_p^k[0,\pi]} = \left( \sum_{j = 0}^k \Bigl\| \frac{d^j u}{dx^j} \Bigr\|_{L_p[0,\pi]}^p \right)^{1/p}.
$$
In particular, $W_2^k[0,\pi]$ ($k = 0, 1, 2, \dots$) are Hilbert spaces and $W_2^0[0,\pi] = L_2[0,\pi]$.

For $\al \ge 0$, denote by $l_2^{\al}$ the space of real infinite sequences $v = \{ v_n \}_{n \ge 1}$ such that
$$
\| v \|_{l_2^{\al}} := \left( \sum_{n = 1}^{\infty} n^{2\al} v_n^2 \right)^{1/2} < \infty.
$$

We define $W_2^{\al}[0,\pi]$ for $\al \in (0,\tfrac{1}{2})$ as the space of functions $u \in L_2[0,\pi]$, whose Fourier coefficients $\hat u_k := \frac{2}{\pi} \int_0^{\pi} u(x) \sin kx \, dx$ form a sequence of $l_2^{\al}$, and $\| u \|_{W_2^{\al}[0,\pi]} := \| \{ \hat u_k \} \|_{l_2^{\al}}$. For brevity, we use the notation $\| . \|_{\al} = \| . \|_{W_2^{\al}[0,\pi]}$ for the norms of functions and $\| . \|_{\al} = \| . \|_{l_2^{\al}}$ for the norms of sequences. Alternatively, the space $W_2^{\al}[0,\pi]$ can be defined by real interpolation $W_2^{\al}[0,\pi] := (L_2[0,\pi], W_2^1[0,\pi])_{\al, 2}$ or by complex interpolation $W_2^{\al}[0,\pi] := [L_2[0,\pi], W_2^1[0,\pi]]_{\al}$ of Banach spaces for $\al \in (0,1)$ (see \cite{Trib78, SS05}).

The following proposition easily follows from the above definitions or
can be deduced from the results of \cite{Trib78}.

\begin{prop} \label{prop:embed}
Let $\al \in (0, \tfrac{1}{2})$. Then, the following assertions hold.

(i) The spaces $W_2^{\al}[0,\pi]$ and $l_2^{\al}$ are compactly embedded in $L_2[0,\pi]$ and $l_2$, respectively.

(ii) The space $W_1^1[0,\pi]$ is continuously embedded in $W_2^{\al}[0,\pi]$.
\end{prop}

Denote by $\mathring{L}_2[0,\pi]$ and $\mathring{W}_2^{\al}[0,\pi]$ the corresponding subspaces of functions with the zero mean value $\int_0^{\pi} u(x) \, dx = 0$. Note that the antiderivative $\sigma(x)$ of the potential $q(x)$ from \eqref{eqvStL} can be chosen up to an additive constant. For definiteness, we assume that $\sigma \in \mathring{L}_2[0,\pi]$.

\subsection{Rational Herglotz-Nevanlinna functions}

A rational Herglotz-Nevanlinna function has the form
\begin{equation} \label{HN}
f(\la) = h_0 \la + h + \sum_{j = 1}^d \frac{\de_j}{h_j - \la},
\end{equation}
where
\begin{equation} \label{condhde}
h_0 \ge 0, \quad h \in \mathbb R, \quad \de_j > 0, \, j = \overline{1,d}, \quad
h_1 < h_2 < \dots < h_d.
\end{equation}

Define the index of $f(\la)$ as 
$$
\ind f := \begin{cases}
                        2d + 1, & h_0 > 0, \\
                        2d, & h_0 = 0.
                   \end{cases}
$$

Additionally, we consider $f = \infty$ with the index $\ind f = -1$.
For $M \ge 0$, denote by $\mathcal R_M$ the set of rational functions of form \eqref{HN} with $\ind f = M$. Put $\mathcal R_{-1} := \{ \infty \}$ and
$$
\mathcal R := \bigcup_{M = -1}^{\infty} \mathcal R_M.
$$

A function of form \eqref{HN} can be represented as a fraction of polynomials:
\begin{equation} \label{ffrac}
f(\la) = \frac{f_{\uparrow}(\la)}{f_{\downarrow}(\la)}, \quad
f_{\downarrow}(\la) := h_0' \prod_{j = 1}^d (h_j - \la), \quad
h_0' := \begin{cases}
            h_0^{-1}, & h_0 > 0, \\
            1, & h_0 = 0.
       \end{cases}
\end{equation}

If $\ind f \ge 2$, then $f(\la)$ has singular points $\{ h_j \}_{j = 1}^d$ and strictly increases on the intervals $(-\infty, h_1)$, $(h_1, h_2)$, \dots, $(h_{d-1}, h_d)$, $(h_d, +\infty)$, since $f'(\la) > 0$. Denote 
$$
\mathring{\pi}(f) := \begin{cases}
                        h_1, & \ind f \ge 2, \\
                        \infty, & \ind f < 2.
                    \end{cases}
$$

Consider two cases separately. 

\smallskip

\underline{Case $\ind f = 2d$}: $\lim\limits_{\la \to \pm \infty} f(\la) = h$, so the polynomial $f_{\uparrow}(\la)$ has exactly $d$ zeros for $h \ne 0$ and $(d-1)$ zeros for $h = 0$. Then, the polynomials $f_{\downarrow}(\la)$ and $f_{\uparrow}(\la)$ can be represented in terms of their coefficients as follows:
$$
f_{\downarrow}(\la) = \sum_{n = 0}^d a_n \la^n, \quad f_{\uparrow}(\la) = \sum_{n = 0}^d b_n \la^n, \quad a_d = (-1)^d, \quad \deg f_{\uparrow} \in \{ d-1, d \}.
$$
Define
$$
c(f) := (a_0, a_1, \dots, a_{d-1}, b_0, b_1, \dots, b_d).
$$

\underline{Case $\ind f = 2d + 1$}: $\lim\limits_{\la \to \pm \infty} f(\la) = \pm \infty$, so the polynomial $f_{\uparrow}(\la)$ has exactly $d + 1$ zeros. Taking \eqref{HN} and \eqref{ffrac} into account, we conclude that
$$
f_{\downarrow}(\la) = \sum_{n = 0}^d a_n \la^n, \quad f_{\uparrow}(\la) = \sum_{n = 0}^{d+1} b_n \la^n, \quad b_{d+1} = (-1)^d, \quad \deg f_{\downarrow} = d.
$$
Define
$$
c(f) := (a_0, a_1, \dots, a_d, b_0, b_1, \dots, b_d).
$$

In the both cases, we associate a rational function $f \in \mathcal R_M$, $M \ge 0$, with the vector $c(f)$ of $\mathbb R^{M + 1}$ and use the Euclidean norm $\| . \| := \| . \|_{\mathbb R^{M + 1}}$ for this vector.

For $\ind f = -1$, we set $f_{\downarrow}(\la) := -1$, $f_{\uparrow}(\la) := 0$, so $c(f)$ is the empty vector and we assume that $\| c(f) \| = 0$.

Consider some examples:
\begin{align*}
& \ind f = 0\colon \quad f(\la) = h, \quad f_{\uparrow}(\la) = h, \quad f_{\downarrow}(\la) = 1, \quad \| c(f) \| = |h|; \\
& \ind f = 1 \colon  \quad f(\la) = h_0 \la + h, \quad f_{\uparrow}(\la) = \la + \frac{h}{h_0}, \quad f_{\downarrow}(\la) = \frac{1}{h_0}, \quad c(f) = \left( \frac{1}{h_0}, \frac{h}{h_0}\right).
\end{align*}

Note that the set $\{ c(f) \colon f \in \mathcal R_M \}$ is not closed with respect to the norm $\| . \|_{\mathbb R^{M+1}}$ for $M \ge 1$. For an integer $M \ge 0$ and reals $Q > 0$, $\de > 0$, define the set $\mathcal R_{M, Q, \de} \subset \mathcal R_M$ of the functions $f(\la)$ of form \eqref{HN} satisfying the conditions:
\begin{gather} \label{RMQde1}
|h| \le Q, \quad \de \le \de_j \le Q, \: j = \overline{1,d}, \quad h_1 \ge 1, \quad h_j + \de \le h_{j+1}, \: j = \overline{1,d-1}, \quad h_d \le Q, \\ \label{RMQde2}
\begin{cases}
\de \le h_0 \le Q, & M = 2d + 1, \\
h_0 = 0, & M = 2d.
\end{cases}
\end{gather}
For $M = -1$, set $\mathcal R_{M, Q, \de} := \mathcal R_{M}$.

\begin{lem} \label{lem:compRM}
Let $M \ge 0$, $Q > 0$, and $\de > 0$. Then the set $\{ c(f) \colon f \in \mathcal R_{M, Q, \de} \}$ is compact in $\mathbb R^{M + 1}$.
\end{lem}

\begin{proof}
By construction, the coefficients $\{ a_n \}$ and $\{ b_n \}$ of the polynomials $f_{\uparrow}(\la)$ and $f_{\downarrow}(\la)$ depend continuously on the constants $h$, $h_0$ (if $h_0 \ne 0$), $h_j$ and $\de_j$ for $j = \overline{1,d}$. Furthermore, the set of vectors $(h, h_0, h_1, \dots, h_d, \de_1, \dots, \de_d)$ described by the conditions \eqref{RMQde1} and \eqref{RMQde2} is compact in $\mathbb R^{2d+2}$. Hence, its image under the continuous mapping is also compact.
\end{proof}

\begin{lem} \label{lem:subset}
Suppose that $M \ge 0$ and $A$ is such a subset of $\mathcal R_M$ that $\{ c(f) \colon f \in A \}$ is compact in $\mathbb R^{M + 1}$ and $\mathring{\pi}(f) \ge 1$ for $f \in A$.
Then, there exist $Q > 0$ and $\de > 0$ such that $A \subset \mathcal R_{M, Q, \de}$.
\end{lem}

\begin{proof}
For $f \in A$, the coefficients $\{ a_n \}$ are uniformly bounded. In the case $M = 2d + 1$, we additionally have $a_d \ne 0$, so $a_d \ge \eps > 0$. This implies \eqref{RMQde2} with $Q = \eps^{-1}$ and some $\de > 0$ for $f \in A$. Consequently, for both even and odd $M$, the roots $\{ h_j \}_{j = 1}^d$ of the polynomial $f_{\downarrow}(\la)$ depend continuously on its coefficients. Furthermore, \eqref{HN} implies
$$
\de_j = -\Res_{\la = h_j} f(\la) = -\frac{f_{\uparrow}(h_j)}{f_{\downarrow}'(h_j)}, \quad h = f(0) - \sum_{j = 1}^d \frac{\de_j}{h_j}.
$$
Thus, the map $c(f) \mapsto (h, h_0, h_1, \dots, h_d, \de_1, \dots, \de_d)$ is continuous. In view of \eqref{condhde}, this yields \eqref{RMQde1} with some positive $Q$ and $\de$ for all $f \in A$.
\end{proof}

Throughout this paper, we use the following notations:
\begin{itemize}
\item When two problems $\mathcal L(\sigma_i, f_i, F_i)$, $i = 1, 2$, of the same class are considered, the lower index $i$ denotes an object related to the corresponding problem.
\item In estimates, the notation $C(A_1, A_2, \dots)$ is used for various positive constants depending on parameters $A_1$, $A_2$, etc. (e.g., $C(Q, \de)$, $C(R, \eps)$).
\end{itemize}

\section{Main results} \label{sec:main}

Consider the problem $\mathcal L(\sigma, f, F)$ of form \eqref{eqv}--\eqref{bc} and introduce the spectral data in accordance with \cite{Gul19}.

Denote by $\vv(x, \la)$ the solution of equation \eqref{eqv} under the initial conditions
\begin{equation} \label{icvv}
\vv(0, \la) = f_{\downarrow}(\la), \quad \vv^{[1]}_{\sigma}(0,\la) = -f_{\uparrow}(\la).
\end{equation}

The boundary value problem \eqref{eqv}--\eqref{bc} is self-adjoint. It has the countable set of real and simple eigenvalues $\la_1 < \la_2 < \dots < \la_n < \la_{n + 1} < \dots$, which coincide with the zeros of the characteristic function 
\begin{equation} \label{defchi}
\chi(\la) := F_{\uparrow}(\la) \vv(\pi, \la) - F_{\downarrow}(\la) \vv^{[1]}_{\sigma}(\pi, \la).
\end{equation}

In addition, define the norming constants
\begin{equation} \label{defga}
\ga_n := \int_0^{\pi} \vv^2(x, \la_n) \, dx + f'(\la_n) \vv^2(0, \la_n) + F'(\la_n) \vv^2(\pi, \la_n), \quad n \ge 1.
\end{equation}

The numbers $\{ \la_n, \ga_n \}_{n \ge 1}$ are called the spectral data of the problem \eqref{eqv}--\eqref{bc}.

This paper is focused on the two spectral problems:

\begin{dirp}
Given $(\sigma, f, F)$, find $\{ \la_n, \ga_n \}_{n \ge 1}$.
\end{dirp}

\begin{ip}
Given $\{ \la_n, \ga_n \}_{n \ge 1}$, find $(\sigma, f, F)$.
\end{ip}

The uniqueness of the inverse problem solution for $\sigma \in \mathring{L}_2[0,\pi]$ was established in \cite{Gul19}. Furthermore, for $\sigma \in \mathring{W}_2^{\alpha}[0,\pi]$, we get the following result.

\begin{thm}[Spectral data characterization] \label{thm:nsc}
Let $\al \in [0,\tfrac{1}{2})$ be real and $M, N \ge -1$ be integers.
For real numbers $\{ \la_n, \ga_n \}_{n \ge 1}$ to be the spectral data of a problem $\mathcal L(\sigma, f, F)$ with $\sigma \in \mathring{W}_2^{\alpha}[0,\pi]$, $f \in \mathcal R_M$, and $F \in \mathcal R_N$, the following conditions are necessary and sufficient:
\begin{gather}
\la_n < \la_{n + 1}, \quad \ga_n > 0, \quad n \ge 1, \\ \label{asympt}
\sqrt{\la_n} = n - \frac{M + N}{2} - 1 + \kappa_n, \qquad \ga_n = \frac{\pi}{2} n^{2M} (1 + \be_n), \quad n \ge 1,
\end{gather}
where $\{ \kappa_n \}$ and $\{ \be_n \}$ belong to $l_2^{\alpha}$.
\end{thm}

Note that, if the sequences $\{ \la_n \}$ and $\{ \ga_n \}$ satisfy the asymptotics \eqref{asympt}, then one can uniquely determine $M$ and $N$, and then find $\{ \kappa_n \}$ and $\{ \be_n \}$.

For $\alpha = 0$, Theorem~\ref{thm:nsc} was proved in \cite{Gul19}. 
For $\alpha \in (0, \tfrac{1}{2})$, the proof is analogous. We only need to establish the correspondence between the spaces $W_2^{\alpha}[0,\pi]$ for $\sigma$ and $l_2^{\alpha}$ for the remainder sequences $\{ \kappa_n \}$ and $\{ \be_n \}$. 
It is worth mentioning that, for the cases of the Dirichlet and the Neumann-Dirichlet boundary conditions ($M \in \{ -1, 0 \}$, $N = -1$) and $\al > 0$, the spectral data characterization has been obtained in \cite{HM06, SS10}. Theorem~\ref{thm:nsc} is consistent with those results for $\al \in (0, \tfrac{1}{2})$ and generalizes them to the case of rational Herglotz-Nevanlinna functions in the boundary conditions.

In order to formulate our main results on the uniform stability, we need to introduce metric spaces of problem parameters and of spectral data.
Consider the metric space 
$$
\mathcal P^{\al, M, N} := \mathring{W}_2^{\al}[0,\pi] \times \mathcal R_M \times \mathcal R_N 
$$
of triples $P = (\sigma, f, F)$ with the distance
\begin{equation} \label{metrd}
\mathbf{d}_{\al}(P_1, P_2) := \| \sigma_1 - \sigma_2 \|_{\al} + \| c(f_1) - c(f_2) \|_{\mathbb R^{M + 1}} + \| c(F_1) - c(F_2) \|_{\mathbb R^{N + 1}}, 
\end{equation}
where $P_i := (\sigma_i, f_i, F_i)$, $i = 1, 2$. Furthermore, introduce the metric space of real-valued sequences $\mathfrak S = \{ \la_n, \ga_n\}_{n \ge 1}$ satisfying the asymptotics \eqref{asympt} for fixed $\al \in [0,\tfrac{1}{2})$, $M, N \ge -1$, with the distance
\begin{equation} \label{metrS}
\varrho_{\al}(\mathfrak S_1, \mathfrak S_2) := \| \{ \kappa_{n,1} - \kappa_{n, 2} \} \|_{\al} + \| \{ \be_{n,1} - \be_{n, 2} \} \|_{\al},
\end{equation}
where $\{ \kappa_{n,i} \}$ and $\{ \be_{n,i} \}$ are the remainders from the asymptotics \eqref{asympt} for $\mathfrak S_i := \{ \la_{n,i}, \ga_{n,i} \}$, $i = 1, 2$.

Next, let us define sets in the introduced metric spaces, on which the uniform stability of direct and inverse spectral mappings will be established.
Fix $\al \in (0, \tfrac{1}{2})$ and $M, N \ge -1$ such that $M + N$ is even (see Remark~\ref{rem:even} below).
For $Q > 0$ and $\de > 0$, introduce the set
\begin{equation} \label{defPQde}
\mathcal P_{Q, \de}^{\alpha, M, N} := \bigl\{ P = (\sigma, f, F) \colon \sigma \in \mathring{W}_2^{\al}[0,\pi], \, \| \sigma \|_{\al} \le Q, \, f \in \mathcal R_{M, Q, \de}, \, F \in \mathcal R_{N, Q, \de}, \, \lambda_1(P) \ge 1\bigr\}.
\end{equation}
We impose the technical restriction $\la_1(P) \ge 1$ according to the previous studies \cite{SS10, Hryn11}, since it simplifies formulations. In general, the spectrum of the boundary value problem $\mathcal L(P)$ is bounded from below.

For $R > 0$ and $\eps > 0$,
denote by $\mathcal B_{R, \eps}^{\al, M, N}$ the set of sequences $\mathfrak S = \{ \la_n, \ga_n \}_{n \ge 1}$ satisfying the asymptotics \eqref{asympt} and the additional requirements:
\begin{gather} \label{reqBla}
\la_1 \ge 1, \quad \sqrt{\la_{n+1}} - \sqrt{\la_n} \ge \eps, \: n \ge 1, \quad \| \{ \kappa_n \}_{n \ge 1} \|_{\al} \le R, \\ \label{reqBga}
1 + \be_n \ge \eps, \: n \ge 1, \quad \| \{ \be_n \}_{n \ge 1} \|_{\al} \le R.
\end{gather}

For brevity, we sometimes omit the indices $\al, M, N$: $\mathcal P_{Q, \de} := \mathcal P_{Q, \de}^{\al, M, N}$, $\mathcal B_{R, \eps} := \mathcal B_{R, \eps}^{\al, M, N}$.

Denote by $\mathscr S$ the spectral transform that maps $P = (\sigma, f, F)$ to the corresponding spectral data $\mathfrak S = \{ \la_n, \ga_n \}_{n \ge 1}$. Our main results on the uniform stability are formulated as follows.

\begin{thm}[Uniform stability for the direct problem] \label{thm:unidir}
Let $Q > 0$ and $\de > 0$. Then $\mathscr S$ maps $\mathcal P_{Q, \de}$ into $\mathcal B_{R, \eps}$, where $R > 0$ and $\eps > 0$ depend on $Q$ and $\de$. Moreover, the direct spectral transform $\mathscr S$ is Lipschitz continuous on $\mathcal P_{Q, \de}$, that is,
\begin{equation} \label{unidir}
\varrho_{\al}(\mathfrak S_1, \mathfrak S_2) \le C(Q, \de) \, \mathbf{d}_{\al} (P_1, P_2)
\end{equation}
for any $P_i \in \mathcal P_{Q, \de}$, $\mathfrak S_i := \mathscr S (P_i)$, $i = 1, 2$.
\end{thm}

\begin{thm}[Uniform stability for the inverse problem] \label{thm:uniinv}
Let $R > 0$ and $\eps > 0$. Then $\mathscr S^{-1}$ maps $\mathcal B_{R, \eps}$ into $\mathcal P_{Q, \de}$, where $Q > 0$ and $\de > 0$ depend on $R$ and $\eps$.
Moreover, the inverse spectral transform $\mathscr S^{-1}$ is Lipschitz continuous on $\mathcal B_{R, \eps}$, that is,
\begin{equation} \label{uniinv}
\mathbf{d}_{\al} (P_1, P_2) \le C(R, \eps) \varrho_{\al}(\mathfrak S_1, \mathfrak S_2)
\end{equation}
for any $\mathfrak S_i \in \mathcal B_{R, \eps}$, $P_i := \mathscr S^{-1} (\mathfrak S_i)$, $i = 1, 2$.
\end{thm}

Theorems~\ref{thm:unidir} and~\ref{thm:uniinv} generalize the results of \cite{SS10} to the problem \eqref{eqv}--\eqref{bc} with rational Herglotz-Nevanlinna functions in the boundary conditions. Their proofs rely on the Lipschitz continuity of the Darboux-type transforms constructed in \cite{Gul19}. Note that the parameter $\de > 0$ is unimportant for $M, N \in \{ -1, 0 \}$. Therefore, this constraint does not appear in the previous studies \cite{SS10, SS13, Hryn11} for constant coefficients.  

Theorems~\ref{thm:unidir} and~\ref{thm:uniinv} establish the unconditional uniform stability for the direct and the inverse problems, respectively. This means that the constants in the estimates \eqref{unidir} and \eqref{uniinv} depend on the constraints for the initial data of the corresponding problem (i.e. on $P$ in \eqref{unidir} and on $\mathfrak S$ in \eqref{uniinv}). Using Theorems~\ref{thm:unidir} and~\ref{thm:uniinv} together, we get the following corollary on the conditional uniform stability for the direct problem and the inverse problem. Therein, the constants in the estimates depend on the constraints for the solution.

\begin{cor} \label{cor:uni}
(i) For any $P_i \in \mathcal P_{Q,\de}$ and $\mathfrak S_i := \mathscr S(P_i)$, $i = 1, 2$, there holds
$$
\mathbf{d}_{\al} (P_1, P_2) \le C(Q, \de) \varrho_{\al}(\mathfrak S_1, \mathfrak S_2).
$$
(ii) For any $\mathfrak S_i \in \mathcal B_{R,\eps}$ and $P_i := \mathscr S^{-1}(\mathfrak S_i)$, there holds
$$
\varrho_{\al}(\mathfrak S_1, \mathfrak S_2) \le C(R, \eps) \, \mathbf{d}_{\al} (P_1, P_2).
$$
\end{cor}

Suppose that we know a finite number of spectral data $\{ \la_n, \ga_n \}_{n = 1}^m$ for the problem $\mathcal L(P)$, $P = (\sigma, f, F) \in \mathcal P^{\al, M, N}$. Then, we can complete them as follows:
\begin{equation} \label{defsdm}
\la_{n,m} := \begin{cases}
                \la_n, & n \le m, \\
                (n - \frac{M+N}{2} - 1)^2, & n > m, 
            \end{cases}
\quad 
\ga_{n,m} := \begin{cases}
                \ga_n, & n \le m, \\
                \frac{\pi}{2} n^{2M}, & n > m,
            \end{cases} \quad
\mathfrak S_m := \{ \la_{n,m}, \ga_{n,m} \}_{n \ge 1},
\end{equation}
and obtain the finite data approximation $P_m = (\sigma_m, f_m, F_m) := \mathscr S^{-1}(\mathfrak S_m)$ of $P$.

\begin{thm} \label{thm:fin1}
Let $0 < \al_1 < \al_2 < \tfrac{1}{2}$, $M, N \ge -1$ be fixed, and
$Q, \de > 0$. Then, for every $P \in \mathcal P_{Q,\de}^{\al_2, M, N}$, the finite data approximation satisfies the uniform stability estimate
\begin{equation} \label{fin1}
\mathbf{d}_{\al_1}(P, P_m) \le C(Q,\de) m^{\al_1 - \al_2}
\end{equation}
for all sufficiently large $m$.
\end{thm}

Note that the right-hand side of \eqref{fin1} tends to zero as $m \to \infty$ uniformly by $P \in \mathcal P_{Q,\de}$.

Finally, consider the case when the finite spectral data are known with an error at most $\eps > 0$:
\begin{equation} \label{esteps}
|\tilde \la_n - \la_n| \le \eps, \quad |\tilde \ga_n - \ga_n| \le \eps, \quad n = \overline{1,m}.
\end{equation}
Then, find the finite data approximation $\tilde P_m := \mathscr S^{-1}(\tilde{\mathfrak S}_m)$ by using the completion $\tilde {\mathfrak S}_m = \{ \tilde \la_{n,m}, \tilde \ga_{n,m} \}_{n \ge 1}$:
$$
\tilde \la_{n,m} := \begin{cases}
                \tilde \la_n, & n \le m, \\
                (n - \frac{M+N}{2} - 1)^2, & n > m, 
            \end{cases}
\quad 
\tilde \ga_{n,m} := \begin{cases}
                \tilde \ga_n, & n \le m, \\
                \frac{\pi}{2} n^{2M}, & n > m.
            \end{cases}
$$

We obtain the following theorem about the uniform stability of the approximation $\tilde P_m$.

\begin{thm} \label{thm:fin2}
Let $0 < \al_1 < \al_2 < \tfrac{1}{2}$, $M, N \ge -1$ be fixed, and
$Q, \de > 0$. Then, there exists $\eps_0 > 0$ such that, for every $P \in \mathcal P_{Q,\de}^{\al_2, M, N}$, $\eps \in (0,\eps_0]$, and all sufficiently large $m$, the following estimate holds:
\begin{equation} \label{fin2}
\mathbf d_{\al_1}(P, \tilde P_m) \le C(Q,\de) \left( \eps m^{\al_1 + \frac{1}{2} - \min\{ 1, 2 M \}} + m^{\al_1 - \al_2} \right). 
\end{equation}
\end{thm}

\section{Auxiliary lemmas} \label{sec:aux}

In this section, we prove several lemmas on the Lipschitz continuity by $P \in \mathcal P_{Q,\de}$ for some auxiliary values.

\begin{lem} \label{lem:unisol}
Let $y(x, \la)$ be the solution of equation \eqref{eqv} satisfying the initial conditions
\begin{equation} \label{icy}
y(x_0) = a, \quad y^{[1]}_{\sigma}(x_0) = b, \quad x_0 \in [0,\pi].
\end{equation}
Then, the maps $(\sigma, a, b, \la) \mapsto y$ and $(\sigma, a, b, \la) \mapsto y^{[1]}_{\sigma}$ from $L_2[0,\pi] \times \mathbb R^3$ to $C[0, \pi]$ are analytic and Lipschitz continuous on bounded sets.
\end{lem}

\begin{proof}
The initial value problem \eqref{eqv}, \eqref{icy} is equivalent to the first-order system
$$
\frac{d}{dx}
\begin{bmatrix}
y \\ y^{[1]}_{\sigma}
\end{bmatrix} = 
\begin{bmatrix}
\sigma & 1 \\
-(\sigma^2 + \la) & -\sigma
\end{bmatrix} \begin{bmatrix}
y \\ y^{[1]}_{\sigma}
\end{bmatrix}, \quad
\begin{bmatrix}
y(x_0) \\ y^{[1]}_{\sigma}(x_0)
\end{bmatrix} = 
\begin{bmatrix}
a \\ b 
\end{bmatrix}.
$$
Reducing this system to an integral equation and solving it by iterations yield the claim. Note that the continuity for solutions of equation \eqref{eqv} with respect to $\sigma$ has been proved in \cite{SS99}.
\end{proof}

\begin{lem} \label{lem:Pcomp}
The set $\mathcal P_{Q,\de}^{\al, M, N}$ for $\al > 0$ is compact in $\mathcal P^{0,M,N}$.
\end{lem}

\begin{proof}
By virtue of part (i) of Proposition~\ref{prop:embed}, the ball $\{ \sigma \in \mathring{W}_2^{\al}[0,\pi] \colon \| \sigma \|_{\al} \le Q \}$ is compact in $L_2[0,\pi]$. This fact together with Lemma~\ref{lem:compRM} and the definitions \eqref{metrd} and \eqref{defPQde} yield the claim.
\end{proof}

Note that $\mathcal P_{Q,\de}^{\al, M, N}$ is not compact in $\mathcal P^{\al, M, N}$. The compact embedding $W_2^{\al}[0,\pi] \subset L_2[0,\pi]$ is crucial.

\begin{lem} \label{lem:la0}
Let $Q > 0$ and $\de > 0$ be fixed. For any $P_i \in \mathcal P_{Q, \de}$, $i = 1, 2$, there holds
\begin{equation} \label{unila0}
|\la_1(P_1) - \la_1(P_2)| + |\ga_1(P_1) - \ga_1(P_2)| \le 
C(Q, \de) \mathbf{d}_0 (P_1, P_2).
\end{equation}
Thus, the maps $\la_1(P)$ and $\ga_1(P)$ are Lipschitz continuous on $\mathcal P_{Q, \de}^{\al, M, N}$ in the metric space $\mathcal P^{0,M,N}$.
\end{lem}

\begin{proof}
According to Lemma~\ref{lem:unisol} and the initial conditions \eqref{icvv}, the maps $(\sigma, c(f), \la) \mapsto \vv(.,\la)$ and $(\sigma, c(f), \la) \mapsto \vv^{[1]}_{\sigma}(., \la)$ are analytic. Hence, the characteristic function $\chi(\la)$ given by \eqref{defchi} is analytic in $\sigma$, $c(f)$, $c(F)$, and $\la$. Therefore, $\chi(\la)$ is Lipschitz continuous for $P = (\sigma, f, F) \in \mathcal P_{Q, \de}$, and $\la$ on bounded sets. Obviously, the first zero $\la_1(P)$ of $\chi(\la)$ is continuous by $P$. 

Let us show the Lipschitz continuity of $\la_1$ using the simplicity of this eigenvalue. Consider a small perturbation $\tilde P = (\tilde \sigma, \tilde f, \tilde F)$ of $P = (\sigma, f, F)$ in $\mathcal P_{Q, \de}$. The Taylor formula implies
$$
\chi(\tilde \la_1) = \chi(\la_1) + \chi'(\la_1)(\tilde \la_1 - \la_1) + \frac{1}{2} \chi''(\xi) (\tilde \la_1 - \la_1)^2, 
$$
where $\xi = (1 - \theta) \la_1 + \theta \tilde \la_1$, $\theta \in (0,1)$. Taking $\chi(\la_1) = \tilde \chi(\tilde \la_1) = 0$ into account, we derive
\begin{equation} \label{sm1}
(\tilde \la_1 - \la_1) \left( \chi'(\la_1) + \frac{1}{2} \chi''(\xi)(\tilde \la_1 - \la_1)\right) = \chi(\tilde \la_1) - \tilde \chi(\tilde \la_1).
\end{equation}

The Lipschitz continuity of $\chi(\la)$ implies
\begin{equation} \label{sm2}
|\chi(\tilde\la_1) - \tilde \chi(\tilde \la_1)| \le C(Q, \de) \mathbf{d}_0 (P, \tilde P).
\end{equation}

Since $\chi(\la)$ is analytic and $\la_1(P)$ is continuous, then $\chi'(\la_1)$ is a continuous function on the set $\mathcal P_{Q, \de}$, which is compact according to Lemma~\ref{lem:Pcomp}. Due to the simplicity of $\la_1$, we have $\chi'(\la_1) \ne 0$, so 
\begin{equation} \label{estchi1}
|\chi'(\la_1)| \ge \eps > 0, \quad \eps = \eps(Q, \de).
\end{equation}
Analogously, we get 
\begin{equation} \label{estchi2}
|\chi''(\xi)| \le C(Q, \de), \quad \text{for}\:\: |\xi| \le R(Q,\de).
\end{equation}

Combining \eqref{sm1}, \eqref{sm2}, \eqref{estchi1}, and \eqref{estchi2}, we arrive at the estimate
\begin{equation} \label{estla0}
|\la_1 - \tilde \la_1| \le C(Q, \de) \, \mathbf{d}_0 (P, \tilde P).
\end{equation}

Remark~2.2 in \cite{Gul19} implies that $\la_1 < \min \{ \mathring{\pi}(f), \mathring{\pi}(F) \}$. This together with \eqref{ffrac} yield $f_{\downarrow}(\la_1) > 0$ and $F_{\downarrow}(\la_1) > 0$. Since $f_{\downarrow}(\la_1(P))$ and $F_{\downarrow}(\la_1(P))$ are continuous on $\mathcal P_{Q,\de}$, then $f_{\downarrow}(\la_1) \ge \eps$ and $F_{\downarrow}(\la_1) \ge \eps$, where $\eps = \eps(Q, \de)$. The latter estimates together with \eqref{defga}, \eqref{estla0}, and the Lipschitz continuity for $\vv(., \la)$ imply the Lipschitz continuity for $\ga_1(P)$, which concludes the proof.
\end{proof}

\begin{remark} \label{rem:lan}
The assertion similar to Lemma~\ref{lem:la0} holds for $\la_n(P)$ and $\ga_n(P)$ for each fixed index $n \ge 2$.
\end{remark}

For reals $\mu, \tau, \rho$ and a rational function $f \in \mathcal R$, consider the transform $\Theta \colon (\mu, \tau, \rho, f) \mapsto \widehat f$, which was defined in \cite{Gul19} by the formula
$$
\widehat f(\la) := \frac{\mu - \la}{f(\la) - \tau} + \rho.
$$

The properties of this transform are described as follows.

\begin{prop}[\hspace*{-3pt}\cite{Gul19}, Lemma 3.1] \label{prop:Theta}
Suppose that $\mu < \mathring{\pi}(f)$ and $\tau \ge f(\mu)$ if $\ind f \ge 0$. Then $\widehat f \in \mathcal R$. Moreover, if $\tau = f(\mu)$, then $\ind \widehat f = \ind f - 1$ and
\begin{equation} \label{Theta1}
\widehat f_{\uparrow}(\la) = \frac{\rho f_{\uparrow}(\la) - (\la - \mu + \tau \rho) f_{\downarrow}(\la)}{\la - \mu}, \quad \widehat f_{\downarrow}(\la) = \frac{f_{\uparrow}(\la) - \tau f_{\downarrow}(\la)}{\la - \mu};
\end{equation}
if $\tau > f(\mu)$, then $\ind \widehat f = \ind f + 1$ and
\begin{equation} \label{Theta2}
\widehat f_{\uparrow}(\la) = -\rho f_{\downarrow}(\la) + (\la - \mu + \tau \rho) f_{\downarrow}(\la), \quad \widehat f_{\downarrow}(\la) = -f_{\uparrow}(\la) + \tau f_{\downarrow}(\la).
\end{equation}
\end{prop}

\begin{lem} \label{lem:uniTheta}
Suppose that $\widehat f = \Theta(\mu, \tau, \rho, f)$, where $f \in \mathcal R_{M, Q, \de}$ for fixed $M$, $Q$, and $\de$, parameters $(\mu, \tau, \rho)$ belong to a compact set satisfying the hypothesis of Proposition~\ref{prop:Theta} and either $\tau = f(\mu)$ or $\tau > f(\mu)$. Then, the coefficients $c(\widehat f)$ are Lipschitz continuous with respect to $\mu$, $\tau$, $\rho$, and $c(f)$.
\end{lem}

\begin{proof}
For the case $\tau > f(\mu)$, the assertion of the lemma trivially follows from \eqref{Theta2}. For $\tau = f(\mu)$, the Lipschitz continuity is proved by interpolation argument. The interpolation nodes are chosen to be separated from $\mu$ and from each other. 
\end{proof}

\section{Transforms} \label{sec:trans}

In this section, we introduce the Darboux-type transforms that have been constructed in \cite{Gul19} and prove their Lipschitz continuity.

Define the domain
$$
\mathcal P := \mathring{L}_2(0,\pi) \times \mathcal R \times \mathcal R.
$$

For $P = (\sigma, f, F) \in \mathcal P$, $\Lambda \in \mathbb R$, $v \in \mathcal D_{\sigma}$ (the domain $\mathcal D_{\sigma}$ is defined by \eqref{defDsi}), $v(x) \ne 0$ for $x \in [0,\pi]$, consider the transform
$$
\mathcal T \colon (P, \Lambda, v) \mapsto \widehat P = (\widehat \sigma, \widehat f, \widehat F)
$$
defined by the following rules:
\begin{align} \label{hatsigma}
\widehat \sigma & := \sigma - \frac{2 v'}{v} + \frac{2}{\pi} \ln \frac{v(\pi)}{v(0)}, \\ \label{hatf}
\widehat f & := \Theta\left(\Lambda, -\frac{v_{\sigma}^{[1]}(0)}{v(0)}, -\frac{v_{\sigma}^{[1]}(0)}{v(0)} + \frac{2}{\pi} \ln \frac{v(\pi)}{v(0)}, f\right), \\ \label{hatF}
\widehat F & := \Theta\left(\Lambda, \frac{v_{\sigma}^{[1]}(\pi)}{v(\pi)}, \frac{v_{\sigma}^{[1]}(\pi)}{v(\pi)} - \frac{2}{\pi} \ln \frac{v(\pi)}{v(0)}, F\right).
\end{align}

Denote by $\psi(x, \la)$ and $z(x, \la, \rho)$
the solutions of equation \eqref{eqv} satisfying the initial conditions
\begin{equation} \label{icpsi}
\psi(\pi, \la) = F_{\downarrow}(\la), \quad \psi_{\sigma}^{[1]}(\pi, \la) = F_{\uparrow}(\la), \quad z(0,\la, \rho) = 1, \quad z^{[1]}_{\sigma}(0, \la, \rho) = -\rho.
\end{equation}

Introduce the transforms
\begin{align*}
\mathbf{T}_-(P) & := \mathcal T(P, \la_1, \vv(., \la_1)), \quad \dom(\mathbf{T}_-) = \{ P \in \mathcal P \colon f \ne \infty, \, F \ne \infty \}, \\
\mathbf{T}_{-+}(P) & := \mathcal T(P, \la_1-2, \vv(., \la_1-2)), \quad
\dom(\mathbf{T}_{-+}) = \{ P \in \mathcal P \colon f \ne \infty \}, \\
\mathbf{T}_{+-}(P) & := \mathcal T(P, \la_1-2, \psi(., \la_1-2)), \quad
\dom(\mathbf{T}_{+-}) = \{ P \in \mathcal P \colon F \ne \infty \}, \\
\mathbf{T}_+(\mu, \nu, P) & := \mathcal T(P, \mu, u), \quad \dom(\mathbf{T}_+) = \{ (\mu, \nu, P) \colon P \in \mathcal P, \, \mu < \la_1, \, \nu > 0 \},
\end{align*}
where $\la_1 = \la_1(P)$ and 
\begin{equation} \label{defu}
u(x) := z(x, \la_1, \rho), \quad
\rho := \frac{\nu \varkappa + f_{\uparrow}(\mu)(\varkappa f_{\downarrow}(\mu) - f_{\uparrow}(\mu))}{\nu + f_{\downarrow}(\mu)(\varkappa f_{\downarrow}(\mu) - f_{\uparrow}(\mu))}, \quad
\varkappa := -\frac{\psi_{\sigma}^{[1]}(0,\mu)}{\psi(0,\mu)}.
\end{equation}

Our notations for the defined transforms have the following meaning:
\begin{itemize}
\item $\mathbf{T}_-$ removes the first eigenvalue $\la_1$, decreases $\ind f$ and $\ind F$.
\item $\mathbf{T}_{-+}$ does not change the spectrum, decreases $\ind f$, and increases $\ind F$.
\item $\mathbf{T}_{+-}$ does not change the spectrum, increases $\ind f$, and decreases $\ind F$.
\item $\mathbf{T_+}$ adds the eigenvalue $\mu$ with the corresponding norming constant $\nu$, increases $\ind f$ and $\ind F$.
\end{itemize}

More formally, the above properties are described by the following proposition.

\begin{prop}[-\hspace*{-3pt}\cite{Gul19}] \label{prop:T}
The transforms $\mathbf{T}_-$, $\mathbf{T}_{-+}$, $\mathbf{T}_{+-}$, and $\mathbf{T}_+$ are well-defined on their domains as mappings to $\mathcal P$. They change the indices of rational functions in the following way:
\begin{align*}
\mathbf{T}_- \colon \quad \ind \widehat f = \ind f - 1, \quad \ind \widehat F = \ind F - 1, \\
\mathbf{T}_{-+} \colon \quad \ind \widehat f = \ind f - 1, \quad \ind \widehat F = \ind F + 1, \\
\mathbf{T}_{+-} \colon \quad \ind \widehat f = \ind f + 1, \quad \ind \widehat F = \ind F - 1, \\
\mathbf{T}_+ \colon \quad \ind \widehat f = \ind f + 1, \quad \ind \widehat F = \ind F + 1.
\end{align*}

Moreover, $\mathbf{T}_-$ and $\mathbf{T}_+$ are inverses of each other in the following sense: if $P \in \dom \mathbf{T_-}$ and $\widehat P = \mathbf{T}_-(P)$, then 
\begin{equation} \label{T+-}
\mathbf{T}_+\bigl( \la_1(P), \ga_1(P), \widehat P \bigr) = P,
\end{equation}
and conversely, if $(\mu, \nu, P) \in \dom \mathbf{T_+}$, then
$\mathbf{T}_- \mathbf{T}_+ (\mu, \nu, P) = P$. Furthermore, 
$\mathbf{T}_{-+} = \mathbf{T}_{+-}^{-1}$.

Let the problem $\mathcal L$ has the spectral data $\mathscr S(P) = \{ \la_n, \ga_n \}_{n \ge 1}$. Then
\begin{align} \label{sdT-}
\mathscr S(\mathbf{T}_-(P)) & = \{ \la_n, \ga_n (\la_n - \la_1)^{-1} \}_{n \ge 2}, \\ \nonumber
\mathscr S(\mathbf{T}_{-+}(P)) & = \{ \la_n, \ga_n (\la_n - \la_1 + 2)^{-1} \}_{n \ge 1}, \\ \nonumber
\mathscr S(\mathbf{T}_{+-}(P)) & = \{ \la_n, \ga_n (\la_n - \la_1 + 2) \}_{n \ge 1}, \\ \nonumber
\mathscr S(\mathbf{T}_+(\mu, \nu, P)) & = \{ \mu, \nu \} \cup \{ \la_n, \ga_n (\la_n - \mu) \}_{n \ge 1}.
\end{align}
\end{prop}

Proceed to studying the Lipschitz continuity of the introduced transforms.

\begin{lem} \label{lem:uniT-}
Suppose that $\mathbf{T} \in \{ \mathbf{T}_-, \mathbf{T}_{-+}, \mathbf{T}_{+-} \}$, fixed integers $M, N$ are such that $\mathring{L}_2(0,\pi) \times \mathcal R_M \times \mathcal R_N \subset \dom \mathbf{T}$, and $\al \in (0, \frac{1}{2})$ is fixed. Then $\mathbf{T}$ maps $\mathcal P_{Q, \de}^{\al, M, N}$ into $\mathcal P_{\widehat Q, \widehat \de}^{\al, \widehat M, \widehat N}$, where $\widehat M := \ind \widehat f$, $\widehat N := \ind \widehat F$, $\widehat Q > 0$ and $\widehat \de > 0$ depend on $Q$ and $\de$. Moreover, 
\begin{equation} \label{uniT-}
\mathbf{d}_{\al} (\widehat P_1, \widehat P_2) \le C(Q, \de) \mathbf{d}_{\al} (P_1, P_2),
\end{equation}
for any $P_i \in \mathcal P_{Q, \de}$, $\widehat P_i = \mathbf{T}(P_i)$, $i = 1, 2$.
Thus, the maps $\mathbf{T}_-$, $\mathbf{T}_{-+}$, and $\mathbf{T}_{+-}$ are Lipschitz continuous on $\mathcal P_{Q, \de}$. 
\end{lem}

\begin{proof}
For definiteness, let us prove the lemma for $\mathbf{T} = \mathbf{T}_-$. The proof for $\mathbf{T}_{-+}$ and $\mathbf{T}_{+-}$ is analogous and even simpler. Suppose that $P \in \mathcal P_{Q, \de} \subset \dom \mathbf{T}_-$. By Lemma~\ref{lem:la0}, the first eigenvalue $\la_1(P)$ is Lipschitz continuous on $\mathcal P_{Q, \de}$. Taking \eqref{icvv} and Lemma~\ref{lem:unisol} into account, we conclude that $v(.) = \vv(., \la_1)$ and $v^{[1]}_{\sigma}(.) = \vv^{[1]}_{\sigma}(., \la_1)$ are Lipschitz continuous on $\mathcal P_{Q, \de}$ as maps from $L_2[0,\pi] \times \mathbb R^{M + 1} \times \mathbb R^{N + 1}$ to $C[0,\pi]$. This implies
\begin{equation} \label{estvabove}
|v(x)| + |v^{[1]}_{\sigma}(x)| \le C(Q, \de), \quad \text{for all} \:\: P \in \mathcal P_{Q, \de}, \, x \in [0,\pi].
\end{equation}

The previous argument yields that the map $(x, P) \mapsto v(x)$ is continuous on the set $[0,\pi] \times \mathcal P_{Q, \de}$, which is compact in the sense of Lemma~\ref{lem:Pcomp}. Therefore, the image of this map is compact in $\mathbb R$.

Note that $v(x)$ is the eigenfunction of the problem $\mathcal L(\sigma, f, F)$ corresponding to the first eigenvalue $\la_1$. Since $f, F \ne \infty$, then the oscillation theory implies $v(x) \ne 0$ for $x \in [0,\pi]$ (see, e.g., \cite[Lemma~2.3]{Gul19}). Due to the compactness of the set $\{ v(x) \}$, we obtain the estimate 
\begin{equation} \label{estvbelow}
|v(x)| \ge \eps > 0,  \quad \text{for all} \:\: P \in \mathcal P_{Q, \de}, \, x \in [0,\pi],
\end{equation}
where $\eps = \eps(Q, \de)$.

Next, put $w := v^{[1]}_{\sigma}$ and note that 
\begin{equation} \label{vdir}
v' = w + \sigma v.
\end{equation}
Substituting \eqref{vdir} into \eqref{hatsigma}, we get
\begin{equation} \label{hatsi1}
\widehat \sigma = -\sigma - \frac{2 w}{v} + \frac{2}{\pi} \ln \frac{v(\pi)}{v(0)}.
\end{equation}

Consider two triples $P_i = (\sigma_i, f_i, F_i) \in \mathcal P_{Q,\de}$, $i = 1, 2$. Since $v_i \in \mathcal D_{\sigma_i}$ \eqref{defDsi}, then $v_i$ and $w_i := (v_i)_{\sigma_i}^{[1]}$ belong to $W_1^1[0,\pi]$. Let us show that
\begin{equation} \label{estdifvw}
\| v_1 - v_2 \|_{W_1^1[0,\pi]} \le C(Q, \de) \mathbf{d}_{\al} (P_1, P_2), \quad
\| w_1 - w_2 \|_{W_1^1[0,\pi]} \le C(Q, \de) \mathbf{d}_{\al} (P_1, P_2).
\end{equation}

Recall that
$$
\| u \|_{W_1^1[0,\pi]} = \| u \|_{L_1[0,\pi]} + \| u' \|_{L_1[0,\pi]}, \quad
\| u \|_{C[0,\pi]} = \max_{x \in [0,\pi]} |u(x)|.
$$

Using \eqref{estvabove}, \eqref{vdir}, the Lipschitz continuity of $v(.)$ and $w(.) = v^{[1]}(.)$ by $P \in P_{Q,\de}$, and the continuous embeddings $W_2^{\al}[0,\pi] \subset L_2[0,\pi] \subset L_1[0,\pi]$, we get
\begin{align*}
\| v_1 - v_2 \|_{L_1[0,\pi]} & \le \pi \| v_1 - v_2 \|_{C[0,\pi]} \le C(Q,\de) \mathbf{d}_{\al}(P_1, P_2), \\
\| v_1' - v_2'\|_{L_1[0,\pi]} & \le \| w_1 - w_2 \|_{L_1[0,\pi]} + \| \sigma_1 - \sigma_2 \|_{\al} \| v_1 \|_{C[0,\pi]} + \| \sigma_2 \|_{\al} \| v_1 - v_2 \|_{C[0,\pi]} \\
& \le C(Q,\de) \mathbf{d}_{\al}(P_1, P_2).
\end{align*}
Analogously, we obtain $\| w_1 - w_2 \|_{L_1[0,\pi]} \le C(Q, \de) \mathbf{d}_{\al}(P_1, P_2)$. Using equation \eqref{eqv}, we get
\begin{equation} \label{wdir}
w_i' = -\sigma_i w_i - \sigma_i^2 w_i - \la_1(P_i) v_i, \quad i = 1, 2.
\end{equation}
Using \eqref{wdir}, the Lipschitz continuity of $v(.)$, $w(.)$, and $\la_1$ by $P \in \mathcal P_{Q,\de}$, and the estimate $\| \sigma_i \|_0 \le \| \sigma_i \|_{\al} \le Q$, we derive
\begin{align*}
\| w_1' - w_2' \|_{L_1[0,\pi]} & \le C(Q,\de) \bigl( \| \sigma_1 - \sigma_2 \|_0 + \| w_1 - w_2 \|_{C[0,\pi]} + \| v_1 - v_2 \|_{C[0,\pi]} + |\la_1(P_1) - \la_2(P_2)| \bigr) \\
& \le C(Q,\de) \mathbf{d}_{\al}(P_1, P_2).
\end{align*}

Summarizing the above arguments, we arrive at \eqref{estdifvw}. Thus, $P \mapsto v$ and $P \mapsto w$ are Lipschitz continuous as maps of $\mathcal P_{Q,\de}$ to $W_1^1[0,\pi]$. Recall that $W_1^1[0,\pi]$ is compactly embedded in $W_2^{\al}[0,\pi]$ due to part (ii) of Proposition~\ref{prop:embed}. Therefore, using \eqref{estvabove}, \eqref{estvbelow}, \eqref{hatsi1}, and \eqref{estdifvw}, we conclude that 
\begin{gather} \label{esthatsi}
\widehat \sigma \in W_2^{\al}[0,\pi], \quad \| \widehat \sigma \|_{\al} \le C(Q,\de), \\ \label{difhatsi}
\| \widehat \sigma_1 - \widehat \sigma_2 \|_{\al} \le C(Q,\de) \mathbf{d}_{\al}(P_1, P_2).
\end{gather}
In other words, the map $P \mapsto \widehat \sigma$ is Lipschitz continuous from $P_{Q,\de}$ to $W_2^{\al}[0,\pi]$. The equality $\int_0^{\pi} \widehat \sigma(x) \, dx = 0$ follows from Proposition~\ref{prop:T}.

Proceed to $\widehat f$ and $\widehat F$. By virtue of Proposition~\ref{prop:T}, we have $\widehat f \in \mathcal R_{\widehat M}$ and $\widehat F \in \mathcal R_{\widehat N}$, where $\widehat M := M - 1$ and $\widehat N := N - 1$ for $\mathbf{T} = \mathbf{T}_-$. Furthermore, due to \cite[Remark~2.2]{Gul19} and \eqref{sdT-}, we have
$$
\min \{ \mathring{\pi}(\widehat f), \mathring{\pi}(\widehat F) \} > \la_2 > \la_1 \ge 1.
$$

Note that the coefficients $\Lambda = \la_1$, $\frac{w(0)}{v(0)}$, $\frac{w(\pi)}{v(\pi)}$, and $\ln \frac{v(\pi)}{v(0)}$, which participate in formulas \eqref{hatf} and \eqref{hatF}, are Lipschitz continuous by $P \in \mathcal P_{Q,\de}$ according to the above arguments. Applying Lemma~\ref{lem:uniTheta}, we conclude that the maps $P \mapsto c(\widehat f)$ and $P \mapsto c(\widehat F)$ are Lipschitz continuous from $\mathcal P_{Q,\de}$ to $\mathbb R^{\widehat M + 1}$ and $\mathbb R^{\widehat N + 1}$, respectively. This together with \eqref{difhatsi} and \eqref{metrd} imply the estimate \eqref{uniT-}.
Furthermore, the images of these maps $\{ c(\widehat f) \colon P \in \mathcal P_{Q,\de} \}$ and $\{ c(\widehat F) \colon P \in \mathcal P_{Q,\de} \}$ are compact. 
Therefore, by Lemma~\ref{lem:subset} they are subsets of $\mathcal R_{\widehat M, \widehat Q, \widehat \de}$ and $\mathcal R_{\widehat N, \widehat Q, \widehat \de}$, respectively, where $\widehat Q > 0$ and $\widehat \de > 0$ depend on $Q$ and $\de$.
This together with \eqref{esthatsi} and \eqref{defPQde} imply $\widehat P = \mathbf{T}(P) \in \mathcal P_{\widehat Q, \widehat \de}^{\al, \widehat M, \widehat N}$ and so conclude the proof.
\end{proof}

The following lemma presents the result analogous to Lemma~\ref{lem:uniT-} for $\mathbf{T}_+$.

\begin{lem} \label{lem:uniT+}
Let $\al \in (0, \tfrac{1}{2})$ and integers $M, N \ge -1$ be fixed. Then $\mathbf{T}_+$ maps the set
$$
\mathcal P^+_{Q,\de} := \bigl\{ (\mu, \nu, P) \in \dom \mathbf{T}_+ \colon P \in \mathcal P_{Q,\de}^{\al, M, N}, \, 
1 \le \mu \le \la_1(P) - \de, \, \de \le \nu \le Q \bigr\}, \quad Q > 0, \, \de > 0,
$$
into $\mathcal P_{\widehat Q, \widehat \de}^{\al, \widehat M, \widehat N}$, where $\widehat Q$ and $\widehat \de$ depend on $Q$ and $\de$. Moreover, the map $\mathbf{T}_+$ is Lipschitz continuous on $\mathcal P^+_{Q, \de}$:
$$
\mathbf{d}_{\al} (\widehat P_1, \widehat P_2) \le C(Q,\de) \bigl( |\mu_1 - \mu_2| + |\nu_1 - \nu_2| + \mathbf{d}_{\al} (P_1, P_2)\bigr)
$$
for any $(\mu_i, \nu_i, P_i) \in \mathcal P^+_{Q,\de}$, $\widehat P_i := \mathbf{T}_+(\mu_i, \nu_i, P_i)$, $i = 1, 2$.
\end{lem}

\begin{proof}
It follows from Lemmas~\ref{lem:unisol} and~\ref{lem:la0} together with the initial conditions \eqref{icpsi} that the maps $(\mu, P) \mapsto \psi(0,\mu)$ and $(\mu, P) \mapsto \psi^{[1]}_{\sigma}(0,\mu)$ are Lipschitz continuous on the compact set 
$$
\mathcal G_{Q,\de} := \{ (\mu, P) \colon P \in \mathcal P_{Q,\de}, \, 1 \le \mu \le \la_1(P) - \de \} 
$$
w.r.t. the metric $\mathbf{d}_0(.,.)$ for $P$. In addition, it has been shown in \cite{Gul19} that $\psi(0,\mu) \ne 0$. Consequently, there holds $|\psi(0,\mu)| \ge \eps > 0$ for $(\mu, P) \in \mathcal G_{Q,\de}$, where $\eps = \eps(Q,\de)$. Consequently, the map $(\mu, P) \mapsto \varkappa = -\dfrac{\psi_{\sigma}^{[1]}(0, \mu)}{\psi(0,\mu)}$ is Lipschitz continuous on $\mathcal G_{Q,\de}$. In particular, $|\varkappa| \le C(Q,\de)$. Analogously, we show that $\rho = \rho(\mu, \nu, P)$ defined by \eqref{defu} is Lipschitz continuous on $\mathcal P_{Q,\de}^+$. Note that the denominator of $\rho$ is non-zero according to the construction in \cite{Gul19}. Next, we consider the function $u(x) = z(x, \mu, \rho)$ and conclude that the transforms $(\mu, \nu, P) \mapsto u(.)$ and $(\mu, \nu, P) \mapsto u^{[1]}_{\sigma}(.)$ are Lipschitz continuous as maps from $\mathbb R \times \mathbb R \times \mathcal P^{0, M, N}$ to $C[0,\pi]$. Furthermore, one can show that the function $u(x)$ possesses the same properties on $\mathcal P_{Q,\de}^+$ as the function $v(x)$ on $\mathcal P_{Q,\de}$ in the proof of Lemma~\ref{lem:uniT-} and analogously complete the proof.
\end{proof}

Now, let us study transforms $\mathscr S \mathbf{T} \mathscr S^{-1}$ which map the spectral data $\mathfrak S \mapsto \widehat {\mathfrak S}$ while $\mathbf{T} \colon P \mapsto \widehat P$.

\begin{lem} \label{lem:uniS-}
Suppose that $\mathbf{T} \in \{ \mathbf{T}_-, \mathbf{T}_{-+}, \mathbf{T}_{+-} \}$, fixed integers $M, N$ are such that $\mathcal P^{0,M,N} \subset \dom \mathbf{T}$, and $\al \in (0, \tfrac{1}{2})$ is fixed. Then $\mathscr S \mathbf{T} \mathscr S^{-1}$ maps $\mathcal B_{R,\eps}^{\al, M, N}$ into $\mathcal B_{\widehat R, \widehat \eps}^{\al, \widehat M, \widehat N}$, where $\widehat R > 0$ and $\widehat \eps > 0$ depend on $R$ and $\eps$. Moreover, the transform $\mathscr S \mathbf{T} \mathscr S^{-1}$ is Lipschitz continuous on $\mathcal B_{R,\eps}$:
\begin{equation} \label{uniS-}
\varrho_{\al}(\widehat{\mathfrak S}_1, \widehat{\mathfrak S}_2) \le C(R,\eps) \varrho_{\al}(\mathfrak S_1, \mathfrak S_2)
\end{equation}
for any $\mathfrak S_i \in \mathcal B_{R,\eps}$, $\widehat{\mathfrak S}_i = \mathscr S \mathbf{T} \mathscr S^{-1} (\mathfrak S_i)$, $i = 1, 2$.
\end{lem}

\begin{proof}
For definiteness, consider $\mathbf{T} = \mathbf{T}_-$. The proof for the other cases is analogous. Let $\mathfrak S = \{ \la_n, \ga_n \}_{n \ge 1} \in \mathcal B_{R,\eps}$. By virtue of \eqref{reqBla}, \eqref{reqBga}, and Theorem~\ref{thm:nsc} for $\al = 0$, which has been proved by Guliyev \cite{Gul19}, we conclude that $\mathfrak S$ are the spectral data of some problem $\mathcal L(P)$ with 
$P = (\sigma, f, F) \in \mathcal P^{0,M,N}$.
According to Proposition~\ref{prop:T}, there hold 
$\widehat P = \mathbf{T} (P) \in \mathcal P^{0, \hat M, \hat N}$,
where $\widehat M = M - 1$, $\widehat N = N - 1$.
Moreover, the corresponding spectral data $\mathfrak S = \mathscr S (\widehat P)$ have the form
\begin{equation} \label{relsd}
\widehat \la_n = \la_n, \quad \widehat \ga_n = \frac{\ga_n}{\la_n - \la_1}, \quad n \ge 2.
\end{equation}

Let us show that $\widehat{\mathfrak S} \in \mathcal B_{\widehat R, \widehat \eps}^{\al, \widehat M, \widehat N}$. By virtue of the direct part of Theorem~\ref{thm:nsc} for $\al = 0$, there hold the asymptotics
\begin{equation} \label{asympthat}
\sqrt{\widehat \la_n} = n - \frac{\widehat M + \widehat N}{2} - 2 + \widehat \kappa_n, \quad
\widehat \ga_n = \frac{\pi}{2} (n-1)^{2 \widehat M} (1 + \widehat \be_n), \quad n \ge 2.
\end{equation}
Here, we take the index shift into account. Since $\widehat \la_n = \la_n$, we have $\widehat \kappa_n = \kappa_n$ and so $\{ \widehat \la_n \}$ and $\{ \widehat \kappa_n \}$ fulfill the conditions \eqref{reqBla}. Using \eqref{asympt}, \eqref{relsd}, and \eqref{asympthat}, we derive
\begin{equation} \label{rasympt}
\la_n - \la_1 = n^2 \bigl(1 + O(n^{-1})\bigr), \quad
1 + \widehat \be_n = (1 + \be_n) \bigl(1 + O(n^{-1})\bigr), 
\end{equation}
where the estimates $O(n^{-1})$ are uniform by $\mathfrak S \in \mathcal B_{R, \eps}$. Since $\la_n - \la_1 \ge \eps$ and $1 + \be_n \ge \eps$, then $1 + \widehat \be_n \ge \widehat\eps$. Note that $\{ \tfrac{1}{n} \} \in l_2^{\al}$ for $\al \in (0,\tfrac{1}{2})$. Therefore, \eqref{rasympt} implies $\{ \widehat \be_n \} \in l_2^{\al}$ and $\| \{ \widehat \be_n \} \|_{\al} \le C(R, \eps)$. Thus, we have proved all the requirements \eqref{reqBla} and \eqref{reqBga} for $\widehat S$ with some positive constants $\widehat R, \widehat \eps$ instead of $R, \eps$.

It remains to prove the estimate \eqref{uniS-}. Consider the corresponding collections of the spectral data $\mathfrak S_i$ and $\widehat{\mathfrak S}_i$ for $i = 1, 2$. It follows from \eqref{asympt} that
\begin{equation} \label{estdifla}
|\la_{n,1} - \la_{n,2}| \le C(R) n |\kappa_{n,1} - \kappa_{n,2}|, \quad n \ge 1.
\end{equation}
Using \eqref{asympt}, \eqref{relsd}, \eqref{asympthat}, and \eqref{estdifla}, we derive
$$
|\widehat \be_{n,1} - \widehat \be_{n,2}| \le C(R, \eps) \left( |\be_{n,1} - \be_{n,2}| + \frac{|\kappa_{n,1} - \kappa_{n,2}|}{n} + \frac{|\kappa_{1,1} - \kappa_{1,2}|}{n^2}\right), \quad n \ge 2.
$$
Hence
$$
\| \{ \widehat\be_{n,1} - \widehat\be_{n,2} \} \|_{\al} \le C(R, \eps) \bigl( \| \{ \be_{n,1} - \be_{n,2} \} \|_{\al} + \| \{ \kappa_{n,1} - \kappa_{n,2}\} \|_{\al} \bigr).
$$
The similar estimate for $\{ \widehat \kappa_{n,1} - \widehat \kappa_{n,2} \}$ is obtained trivially. Thus, in view of \eqref{metrS}, we arrive at \eqref{uniS-}.
\end{proof}

\begin{lem} \label{lem:uniS+}
Let $\al \in (0,\tfrac{1}{2})$ and the integers $M, N \ge -1$ be fixed. Then, for any $\mathfrak S \in \mathcal B_{R,\eps}^{\al, M, N}$, $1 \le \mu \le \la_1 - \eps$, and $\eps \le \nu \le R$, the spectral data $\mathscr S \mathbf{T}_+ (\mu, \nu, \mathscr S^{-1}(\mathfrak S))$ belong to $\mathcal B_{\widehat R, \widehat \eps}^{\al, \widehat M, \widehat N}$,
where $\widehat R > 0$ and $\widehat \eps > 0$ depend on $R$ and $\eps$. Furthermore,
the Lipschitz continuity holds:
$$
\varrho_{\al} (\widehat{\mathfrak S}_1, \widehat{\mathfrak S}_2) \le C(R, \eps) \bigl( |\mu_1 - \mu_2| + |\nu_1 - \nu_2| + \varrho_{\al} (\mathfrak S_1, \mathfrak S_2)\bigr)
$$
for any $\mathfrak S_i \in \mathcal B_{R,\eps}$, $1 \le \mu_i \le \la_{1,i} - \eps$, $\eps \le \nu_i \le R$, $\widehat{\mathfrak S}_i := \mathscr S \mathbf{T}_+ (\mu_i, \nu_i, \mathscr S^{-1}(\mathfrak S_i))$, $i = 1, 2$.
\end{lem}

The proof of Lemma~\ref{lem:uniS+} is analogous to Lemma~\ref{lem:uniS-}.

\section{Proofs of the main results} \label{sec:proofs}

In this section, we present the proofs of the main theorems formulated in Section~\ref{sec:main}. Theorems~\ref{thm:unidir} and~\ref{thm:uniinv} are proved by induction on $M + N = 2p$. The induction base follows from the results of \cite{SS10}. At each step, we use the Lipschitz continuity of the Darboux-type transforms discussed in Section~\ref{sec:trans}. Next, we consider the finite data approximation and prove Theorems~\ref{thm:fin1} and~\ref{thm:fin2}.

\begin{proof}[Proof of Theorem~\ref{thm:unidir}]
In the case of the Dirichlet boundary conditions $M = N = -1$, Theorem~\ref{thm:unidir} follows from \cite[Theorem 3.8]{SS10} (see also \cite[Theorem~3.15]{SS13}).

Suppose that Theorem~\ref{thm:unidir} is valid for $M + N = 2p$. Let us prove it for $M + N = 2p + 2$. Consider the case $M \ge 0$, $N \ge 0$. Fix $\al \in (0, \tfrac{1}{2})$, $Q > 0$, and $\de > 0$. 
Let $P \in \mathcal P_{Q, \de}^{\al, M, N}$.
By Lemma~\ref{lem:uniT-}, $\widehat P := \mathbf{T}_-(P)$ belongs to
$\mathcal P_{\widehat Q, \widehat \de}^{\al, \widehat M, \widehat N}$, where $\widehat M = M - 1$, $\widehat N = N - 1$, $\widehat Q$ and $\widehat \de$ depend on $Q$ and $\de$. Moreover, the Lipschitz estimate \eqref{uniT-} is fulfilled. For $\widehat M + \widehat N = 2p$, the assertion of Theorem~\ref{thm:unidir} holds by virtue of the induction hypothesis. Therefore, $\widehat{\mathfrak S} := \mathscr S(\widehat P)$ lies in $\mathcal B_{\widehat R, \widehat \eps}^{\al, \widehat M, \widehat N}$, where $\widehat R$ and $\widehat \eps$ depend on $\widehat Q$ and $\widehat \de$, and
\begin{equation} \label{uni1}
\varrho_{\al} (\widehat {\mathfrak S}_1, \widehat {\mathfrak S}_2) \le C(\widehat Q, \widehat \de) \mathbf{d}_{\al}(\widehat P_1, \widehat P_2), \quad \widehat P_i \in \mathcal P_{\widehat Q, \widehat \de}.
\end{equation}

Next, in view of \eqref{T+-}, the spectral data $\mathfrak S := \mathscr S(P)$ equal $\mathscr S \mathbf{T}_+ (\la_1(P), \ga_1(P), \mathscr S^{-1}(\widehat {\mathfrak S}))$. For applying Lemma~\ref{lem:uniS+}, we have to show that
\begin{equation} \label{sd1}
1 \le \la_1(P) \le \la_2(P) - \eta, \quad \eta \le \ga_1(P) \le \Omega,
\end{equation}
for any $P \in \mathcal P_{Q,\de}$,
where positive constants $\eta$ and $\Omega$ depend in $Q$ and $\de$. The inequality $\la_1(P) \ge 1$ follows from \eqref{defPQde}.
According to Lemmas~\ref{lem:Pcomp} and~\ref{lem:la0} and Remark~\ref{rem:lan}, the functions $\la_2(P) - \la_1(P)$ and $\ga_1(P)$ are continuous on the compact set $\mathcal P_{Q,\de}$ in $\mathcal P^{0,M,N}$. This together with the inequalities $\la_1 < \la_2$ and $\ga_1 > 0$ yield \eqref{sd1}. Consequently, Lemma~\ref{lem:uniS+} implies $\mathfrak S \in \mathcal B_{R, \eps}^{\al, M, N}$ and
\begin{equation} \label{uni2}
\varrho_{\al}(\mathfrak S_1, \mathfrak S_2) \le C \bigl( |\la_1(P_1) - \la_1(P_2)| + |\ga_1(P_1) - \ga_2(P_2)| + \varrho_{\al}(\widehat{\mathfrak S}_1, \widehat{\mathfrak S}_2)\bigr),
\end{equation}
where the constants $R > 0$, $\eps > 0$, and $C$ depend on $\widehat R$, $\widehat \eps$, $\eta$, and $\Omega$. From the above arguments, one can easily see that these constants depend on $Q$ and $\de$.

Thus, we have shown that, for any $P \in \mathcal P_{Q,\de}$, the spectral data $\mathscr S(P)$ belong to $\mathcal B_{R, \eps}$. Combining the Lipschitz estimates \eqref{uniT-}, \eqref{uni1}, \eqref{uni2}, and \eqref{unila0}, we arrive at \eqref{unidir}. So, the assertion of Theorem~\ref{thm:unidir} is proved for $M + N = 2p + 2$, $M \ge 0$, $N \ge 0$. This proof can be presented by the following scheme:
$$
P \: \xrightarrow[\text{Lemma~\ref{lem:uniT-}}]{\text{$\mathbf{T}_-$}} \: \widehat P \: \xrightarrow[\text{induction hypothesis for $\widehat M + \widehat N = M + N - 2$}]{\text{$\mathscr S$}} \: \widehat {\mathfrak S} \: \xrightarrow[\text{Lemma~\ref{lem:uniS+}}]{\text{$\mathbf{T}_+$}} \: \mathfrak S
$$

The cases $M = -1$ and $N = -1$ can be analogously reduced to the previously studied case $\widehat M \ge 0$, $\widehat N \ge 0$, $\widehat M + \widehat N = M + N$ by using the transforms $\mathbf{T}_{-+}$ and $\mathbf{T}_{+-}$:
\begin{align*}
M = -1 \colon & \quad P \: \xrightarrow[\text{Lemma~\ref{lem:uniT-}}]{\text{$\mathbf{T}_{+-}$}} \: \widehat P \: \xrightarrow[\text{induction hypothesis for $\widehat M + \widehat N = M + N$, $\widehat M, \widehat N \ge 0$}]{\text{$\mathscr S$}} \: \widehat {\mathfrak S} \: \xrightarrow[\text{Lemma~\ref{lem:uniS-}}]{\text{$\mathbf{T}_{-+}$}} \: \mathfrak S, \\
N = -1 \colon & \quad P \: \xrightarrow[\text{Lemma~\ref{lem:uniT-}}]{\text{$\mathbf{T}_{-+}$}} \: \widehat P \: \xrightarrow[\text{induction hypothesis for $\widehat M + \widehat N = M + N$, $\widehat M, \widehat N \ge 0$}]{\text{$\mathscr S$}} \: \widehat {\mathfrak S} \: \xrightarrow[\text{Lemma~\ref{lem:uniS-}}]{\text{$\mathbf{T}_{+-}$}} \: \mathfrak S.
\end{align*}

This concludes the proof for $M + N = 2p + 2$. The induction yields the assertion of Theorem~\ref{thm:unidir} for any even $M + N \ge -2$.
\end{proof}

\begin{proof}[Proof of Theorem~\ref{thm:uniinv}]
Let us use induction on $M + N = 2p$. For the case of the Dirichlet boundary conditions $(p = -1)$, Theorem~\ref{thm:uniinv} follows from \cite[Theorem 3.8]{SS10} (see also \cite[Theorem~3.15]{SS13}).

Suppose that the assertion of Theorem~\ref{thm:uniinv} holds for $M + N = 2p$. Let us prove it for $M + N = 2p + 2$. Schematically, the proof for different cases can be presented as follows:
\begin{align*}
M, N \ge 0 \colon & \quad \mathfrak S \: \xrightarrow[\text{Lemma~\ref{lem:uniS-}}]{\text{$\mathbf{T}_-$}} \: \widehat{\mathfrak S} \: \xrightarrow[\text{induction hypothesis for $\widehat M + \widehat N = M + N - 2$}]{\text{$\mathscr S^{-1}$}} \: \widehat P \: \xrightarrow[\text{Lemma~\ref{lem:uniT+}}]{\text{$\mathbf{T}_+$}} \: P,\\
M = -1 \colon & \quad \mathfrak S \: \xrightarrow[\text{Lemma~\ref{lem:uniS-}}]{\text{$\mathbf{T}_{+-}$}} \: \widehat{\mathfrak S} \: \xrightarrow[\text{induction hypothesis for $\widehat M + \widehat N = M + N$, $\widehat M, \widehat N \ge 0$}]{\text{$\mathscr S^{-1}$}} \: \widehat P \: \xrightarrow[\text{Lemma~\ref{lem:uniT-}}]{\text{$\mathbf{T}_{-+}$}} \: P,\\
N = -1 \colon & \quad \mathfrak S \: \xrightarrow[\text{Lemma~\ref{lem:uniS-}}]{\text{$\mathbf{T}_{-+}$}} \: \widehat{\mathfrak S} \: \xrightarrow[\text{induction hypothesis for $\widehat M + \widehat N = M + N$, $\widehat M, \widehat N \ge 0$}]{\text{$\mathscr S^{-1}$}} \: \widehat P \: \xrightarrow[\text{Lemma~\ref{lem:uniT-}}]{\text{$\mathbf{T}_{+-}$}} \: P.
\end{align*}

Consider the case $M, N \ge 0$ in more detail. Let $\al \in (0, \tfrac{1}{2})$, $M$, and $N$ be fixed. Suppose that $\mathfrak S \in \mathcal B_{R, \eps}^{\al, M, N}$ for some $R > 0$ and $\eps > 0$. Put $\widehat{\mathfrak S} := \mathscr S \mathbf{T}_- \mathscr S^{-1} \mathfrak S$. By Lemma~\ref{lem:uniS-}, we have $\widehat{\mathfrak S} \in \mathcal B_{\widehat R, \widehat \eps}^{\al, \widehat M, \widehat N}$, where $\widehat R > 0$ and $\widehat \eps > 0$ depend on $R$ and $\eps$, $\widehat M = M - 1$, $\widehat N = N - 1$. Moreover, the estimate \eqref{uniS-} is valid. Since $\widehat M + \widehat N = M + N - 2$, the map $\mathscr S^{-1} \colon \widehat{\mathfrak S} \mapsto \widehat P$ satisfies the assertion of Theorem~\ref{thm:uniinv} by the induction hypothesis. Hence, $\widehat P \in \mathcal P_{\widehat Q, \widehat \de}^{\al, \widehat M, \widehat N}$ with some $\widehat Q$ and $\widehat \de$ depending on $\widehat R$ and $\widehat \eps$. Additionally, the Lipschitz estimate holds:
\begin{equation} \label{uni3}
\mathbf{d}_{\al}(\widehat P_1, \widehat P_2) \le C(\widehat R, \widehat \eps) \varrho_{\al}(\widehat {\mathfrak S}_1, \widehat {\mathfrak S}_2), \quad \widehat {\mathfrak S}_i \in \mathcal B_{\widehat R, \widehat \eps}.
\end{equation}

Next, note that $P := \mathscr S^{-1} (\mathfrak S)$ equals $\mathbf{T}_+(\la_1, \ga_1, \widehat P)$, where $\la_1$ and $\ga_1$ are from $\mathfrak S$. In order to apply Lemma~\ref{lem:uniT+}, we have to show that
\begin{equation} \label{sd2}
1 \le \la_1 \le \widehat \la_2 - \eta, \quad \eta \le \ga_1 \le \Omega, 
\end{equation}
for any $\mathfrak S \in \mathcal B_{R, \eps}$,
where $\eta$ and $\Omega$ are positive constants depending on $R$ and $\eps$. The estimates \eqref{sd2} readily follow from $\mathfrak S \in \mathcal B_{R, \eps}$ and the equality $\widehat \la_2 = \la_2$. Then, by Lemma~\ref{lem:uniT+}, we get $P \in \mathcal P_{Q, \de}^{\al, M, N}$, where $Q > 0$ and $\de > 0$ depend on $R$ and $\eps$, and
\begin{equation} \label{uni4}
\mathbf{d}_{\al}(P_1, P_2) \le C(R, \eps) \bigl( |\la_{1,1} - \la_{1,2}| + |\ga_{1,1} - \ga_{1,2}| + \mathbf{d}_{\al}(\widehat P_1, \widehat P_2)\bigr).
\end{equation}

Combining the Lipschitz estimates \eqref{uniS-}, \eqref{uni3}, and \eqref{uni4}, we arrive at \eqref{uniinv} and so conclude the proof for $M + N = 2p$, $M \ge 0$, $N \ge 0$. The cases $M = -1$ and $N = -1$ are studied analogously. Induction completes the proof.
\end{proof}

\begin{remark} \label{rem:proofnsc}
The proof of Theorem~\ref{thm:unidir}, in particular, shows by induction that, for any $\sigma \in \mathring{W}_2^{\al}[0,\pi]$, $f \in \mathcal R_M$, and $F \in \mathcal R_N$, the sequences of the remainders $\{ \kappa_n \}$ and $\{ \be_n \}$ from the spectral data asymptotics \eqref{asympt} belong to $l_2^{\al}$. This proves the direct part of Theorem~\ref{thm:nsc}. Analogously, the inverse part follows from the proof of Theorem~\ref{thm:uniinv}.
\end{remark}

\begin{remark} \label{rem:even}
Note that the transforms $\mathbf{T}_-$, $\mathbf{T}_{-+}$, $\mathbf{T}_{+-}$, and $\mathbf{T}_{+}$ do not change the parity of $M + N$. Therefore, we have proved Theorems~\ref{thm:unidir} and~\ref{thm:uniinv} only for even $M + N$, since they are based on the results of \cite{SS10}, which were obtained only for the Dirichlet boundary conditions ($M + N = -2$). Nevertheless, the spectral data characterization has been proved in \cite{HM06} for the Dirichlet-Dirichlet and the Neumann-Dirichlet boundary conditions, so we get Theorem~\ref{thm:nsc} for both even and odd $M + N$.
\end{remark}

Proceed to considering the finite data approximation $P_m$ of $P$.

\begin{proof}[Proof of Theorem~\ref{thm:fin1}]
Let $P$ satisfy the hypothesis of the theorem. In view of the continuous embedding $W_2^{\al_2}[0,\pi] \subset W_2^{\al_1}[0,\pi]$ for $\al_1 < \al_2$, we have $P \in \mathcal P_{Q,\de}^{\al_1, M, N}$. By Theorem~\ref{thm:nsc}, the numbers $\mathfrak S_m = \{ \la_{n,m}, \ga_{n,m} \}_{n \ge 1}$ for each sufficiently large $m$ are the spectral data of some problem $\mathcal L(P_m)$, where $P_m \in \mathcal P^{\al, M, N}$ for any $\al \in [0,\tfrac{1}{2})$. Applying Theorem~\ref{thm:unidir} with $\al = \al_1$, we conclude that $\mathfrak S \in \mathcal B_{R,\eps}$, where $R$ and $\eps$ depend on $Q$ and $\de$. Furthermore, we have $\mathfrak S_m \in \mathcal B_{R,\eps}$ for sufficiently large $m$ in view of \eqref{reqBla}, \eqref{reqBga}, and \eqref{defsdm}. Theorem~\ref{thm:uniinv} implies $P_m \in \mathcal P_{\widehat Q, \widehat \de}$, where $\widehat Q$ and $\widehat \de$ depend on $Q$ and $\de$. Therefore, part (i) of Corollary~\ref{cor:uni} yields
\begin{equation} \label{est1}
\mathbf{d}_{\al_1}(P, P_m) \le C(Q,\de) \varrho_{\al_1}(\mathfrak S, \mathfrak S_m).
\end{equation}
According to \eqref{metrS}, we have
$$
\varrho_{\al_1}(\mathfrak S, \mathfrak S_m) = \| \{ \kappa_n - \kappa_{n,m} \} \|_{\al_1} + \| \{ \be_n - \be_{n,m} \} \|_{\al_1}.
$$
Using \eqref{defsdm}, we get
$$
\| \{ \kappa_n - \kappa_{n,m} \} \|_{\al_1} = \sqrt{\sum_{n = m + 1}^{\infty} n^{2\al_1} \kappa_n^2} \le m^{\al_1 - \al_2} \| \{ \kappa_n \} \|_{\al_2}.
$$
Using the similar estimate for $\{ \be_n \}$ together with \eqref{reqBla} and \eqref{reqBga}, we obtain
\begin{equation} \label{est2}
\varrho_{\al_1}(\mathfrak S, \mathfrak S_m) \le m^{\al_1 - \al_2} \bigl( \| \{ \kappa_n \} \|_{\al_2} + \| \{ \be_n \} \|_{\al_2} \bigr) \le C(Q,\de) m^{\al_1 - \al_2}.
\end{equation}
Combining \eqref{est1} and \eqref{est2} yields the claim.
\end{proof}

\begin{proof}[Proof of Theorem~\ref{thm:fin2}]
Analogously to the previous proof, we get the estimate
\begin{equation} \label{est3}
\mathbf{d}_{\al_1}(P, \tilde P_m) \le C(Q,\de) \varrho_{\al_1}(\mathfrak S, \tilde{\mathfrak S}_m)
\end{equation}
for all sufficiently small $\eps > 0$ and sufficiently large $m$. It follows from \eqref{asympt} and \eqref{esteps} that
$$
|\kappa_n - \tilde \kappa_n| \le C(Q,\de) \eps n^{-1}, \quad 
|\be_n - \tilde \be_n| \le C(Q,\de) \eps n^{-2M}, \quad n = \overline{1,m}.
$$
Using the latter estimates together with \eqref{est2} and \eqref{est3}, we arrive at \eqref{fin2}.
\end{proof}

\begin{remark}
In this paper, we confine ourselves by $\al < \tfrac{1}{2}$ for technical simplicity. Our methods also work for $\al \ge \tfrac{1}{2}$, but then spectral data asymptotics contain more terms:
$$
\sqrt{\la_n} = s(n) + \frac{\om_1}{s(n)} + \frac{\om_2}{s^3(n)} + \dots + \frac{\om_k}{s^{2k-1}(n)} + o\left( n^{-(2k-1)} \right), \quad s(n) = n - \tfrac{M + N}{2} - 1.
$$
The constants $\{ \om_j \}_{j = 1}^k$ are related to $\sigma$ and to the coefficients of the rational functions $f(\la)$ and $F(\la)$, which complicates the technique.
\end{remark}

\medskip

{\bf Acknowledgment.} The author is grateful to Professors Maria A. Kuznetsova and Namig J. Guliyev for reading the manuscript and for their valuable remarks.

\medskip

{\bf Funding.} This work was supported by Grant 24-71-10003 of the Russian Science Foundation, https://rscf.ru/en/project/24-71-10003/.

\noindent Natalia Pavlovna Bondarenko \\

\noindent 1. Department of Mechanics and Mathematics, Saratov State University, 
Astrakhanskaya 83, Saratov 410012, Russia, \\

\noindent 2. Department of Applied Mathematics, Samara National Research University, \\
Moskovskoye Shosse 34, Samara 443086, Russia, \\

\noindent 3. S.M. Nikolskii Mathematical Institute, RUDN University, 6 Miklukho-Maklaya St, Moscow, 117198, Russia, \\

\noindent 4. Moscow Center of Fundamental and Applied Mathematics, Lomonosov Moscow State University, Moscow 119991, Russia.\\

\noindent e-mail: {\it bondarenkonp@sgu.ru}
\end{document}